\documentclass[12pt]{amsart}
\usepackage{amssymb, amsmath}
\usepackage{amsthm}
\usepackage{subfigure}
\usepackage{cite}
\usepackage{enumerate}
\usepackage{cases}
\usepackage[font={it}, margin=1cm]{caption}
\usepackage{enumitem}
\usepackage{tikz}
\usepackage{verbatim}
\usetikzlibrary{decorations.markings, decorations.pathreplacing, positioning,arrows, matrix, decorations.pathmorphing, shapes.geometric, calc}
\usetikzlibrary{backgrounds, decorations}

\reversemarginpar
 
\newtheorem{theorem}{Theorem}[section]
\newtheorem{proposition}[theorem]{Proposition} 
\newtheorem{thmletter}{Theorem}
\newtheorem{corollary}[theorem]{Corollary}
\newtheorem{question}[theorem]{Question}

\newcommand{\p}[1]{\noindent {\newline\bf #1.}}
\newcommand{\out}{\operatorname{Out}}
\newcommand{\aut}{\operatorname{Aut}}
\newcommand{\inn}{\operatorname{Inn}}
\newcommand{\cat}{\operatorname{CAT}}

\title{On a question of Bumagin and Wise}
\author{Alan D. Logan}
\address{School of Mathematics and Statistics\\
University of Glasgow\\
  Glasgow, G12 8QW, UK}
\email{Alan.Logan@glasgow.ac.uk}
\subjclass[2010]{20E06, 20E26, 20F28, 20F67}

\keywords{Residually finite groups, Outer automorphism groups, Recursive presentability}

\begin{document}

\maketitle

\begin{abstract}
Motivated by a question of Bumagin and Wise, we construct a continuum of finitely generated, residually finite groups whose outer automorphism groups are pairwise non-isomorphic finitely generated, non-recursively-presentable groups. These are the first examples of such residually finite groups.
\end{abstract}

\section{Introduction}
In this paper we construct the first examples of finitely generated, residually finite groups $G$ whose outer automorphism groups are finitely generated and not recursively presentable. Indeed, we construct a continuum, so $2^{\aleph_0}$, of such groups $G$ with pairwise non-isomorphic outer automorphism groups. Our construction is motivated by a question of Bumagin and Wise, who asked if every countable group $Q$ could be realised as the outer automorphism group of a finitely generated, residually finite group $G_Q$.
In this paper we solve a finite-index version of this question for $Q$ finitely generated and residually finite, and the aforementioned result then follows. Bumagin and Wise solved the question for $Q$ finitely presented \cite{BumaginWise2005}. In previous work, the author gave a partial solution for $Q$ finitely generated and recursively presentable \cite[Theorem A]{logan2015outer}, and a complete solutions for these groups assuming that there the exists a ``malnormal version'' of Higman's embedding theorem \cite[Theorem B]{logan2015outer}.

\p{Residually finite groups}
A group $G$ is residually finite if for all $g\in G\setminus\{1\}$ there exists a homomorphism image $\phi_g:G\rightarrow F_g$ where $F_g$ is finite and where $\phi_g(g)\neq1$.
Residual finiteness is a strong finiteness property. For example, finitely presentable, residually finite groups have soluble word problem, while finitely generated, residually finite groups are Hopfian \cite{Malcev1940}.
Our main result, which is Theorem~\ref{corol:mainresult}, contrasts with these ``nice'' properties as it implies that finitely generated groups can have very complicated symmetries.

Fundamental to this paper is the existence of finitely generated, residually finite groups which are not recursively presentable. Bridson-Wilton \cite[Section 2]{bridson2015triviality} point out that the existence of such groups follows from work of Slobodsko\v\i \cite{Slobodskoi1981undecidability}. The ``continuum'' statement in the main result, Theorem~\ref{corol:mainresult}, relies on the fact that there is a continuum of such groups, which is due to Minasyan-Ol'shanskii-Sonkin \cite[Theorem 4]{minasyan2009periodic}.
To see that the existence of such groups is fundamental to our argument, suppose that every finitely generated, residually finite group is recursively presentable, and let $G$ be a finitely generated, residually finite group with finitely generated outer automorphism group. Then $\aut(G)$ is finitely generated and residually finite \cite{Baumslag}, and hence is recursively presentable. Therefore, as the kernel of $\aut(G)\rightarrow \out(G)$ is finitely generated (because $\inn(G)\cong G/Z(G)$), $\out(G)$ is also recursively presentable. Hence, the existence of finitely generated, residually finite groups which are not recursively presentable is necessary for our argument.

\p{The main construction}
The main result of this paper, the result stated in the abstract, is Theorem~\ref{corol:mainresult}. This theorem follows from a more general construction, Theorem~\ref{thm:maintheorem}, which relates to the outer automorphism groups of HNN-extensions of
certain groups. Theorem~\ref{thm:maintheorem} yields the following two corollaries, each of which individually solves Bumagin and Wise's question up to finite index for $Q$ finitely generated and residually finite. A {triangle group} $T_{i, j, k}:=\langle a, b; a^i, b^j, (ab)^k\rangle$ is called hyperbolic if $i^{-1}+j^{-1}+k^{-1}<1$.

\begin{corollary}
\label{corol:triangle}
Fix a hyperbolic triangle group $H:=T_{i, j, k}$. Then every finitely-generated group $Q$ can be embedded as a finite index subgroup of the outer automorphism group of an HNN-extension $G_Q$ of $H$, where $G_Q$ is residually finite if $Q$ is residually finite.
\end{corollary}

The following corollary is satisfied by a random group, in the sense of Gromov~\cite{gromov1996geometric}~\cite{ollivier2005january}, at density $<1/6$ \cite{dahmani2011random}~\cite{ollivier2011cubulating}.

\begin{corollary}
\label{corol:random}
Fix a hyperbolic group $H$ which has Serre's property FA and which acts properly and cocompactly on a $\cat(0)$ cube complex. Then every finitely-generated group $Q$ can be embedded as a finite index subgroup of the outer automorphism group of an HNN-extension $G_Q$ of $H$, where $G_Q$ is residually finite if $Q$ is residually finite.
\end{corollary}

The main result of the paper is the following. By a \emph{continuum} we mean a set of cardinality $2^{\aleph_0}$, that is, of cardinality equal to that of the real numbers $\mathbb{R}$.

\begin{thmletter}
\label{corol:mainresult}
There exists a continuum of finitely generated, residually finite groups whose outer automorphism groups are pairwise non-isomorphic finitely generated, non-recursively-presentable groups.
\end{thmletter}

We prove Theorem~\ref{corol:mainresult} by noting the existence of a continuum of finitely generated, residually finite groups which are not recursively presentable, and then apply Theorem~\ref{thm:maintheorem} (or rather, either of the above corollaries) to these groups.

\p{Outline of the paper}
In Section~\ref{sec:prelim} we give two preliminary results on a certain class of HNN-extensions, which we call ``inner'' HNN-extensions. These are Theorem~\ref{thm:Prelim1}, which describes a certain subgroup of the outer automorphism group of an inner HNN-extension, and Proposition~\ref{thm:Prelim2}, which classifies the residual finiteness of a certain class of inner HNN-extensions.
In Section~\ref{sec:main} we prove our main results, Theorems~\ref{corol:mainresult}~and~\ref{thm:maintheorem}.
In Section~\ref{sec:Qs} we prove a result for finitely presented (rather than finitely generated) residually finite groups.

\p{Acknowledgments}
The author would like to thank Steve Pride and Tara Brendle for many helpful discussions about the work surrounding the paper, and Henry Wilton for ideas which led to Corollary~\ref{corol:random}.

\section{Two preliminary results}
\label{sec:prelim}

Our construction of Theorem~\ref{thm:maintheorem}, which leads to the main result, applies two preliminary results on \emph{inner HNN-extensions}, which are HNN-extensions where the action of the stable letter on the associated subgroup(s) is an inner automorphism of the base group. Such an HNN-extension $G$ has the following form (up to isomorphism).
\[
G\cong \langle H, t; k^t=k, k\in K\rangle
\]
The first result of this section, Theorem~\ref{thm:Prelim1}, relates to the outer automorphism groups of inner HNN-extensions, while the second result, Proposition~\ref{thm:Prelim2}, relates to their residual finiteness.

\p{First preliminary result} The first preliminary result, Theorem~\ref{thm:Prelim1}, tells us about a subgroup of the outer automorphism group of an inner HNN-extension. This subgroup, denoted $\out^H(G)$, is the subgroup which consists of those outer automorphisms $\Phi$ with a representative $\phi\in\Phi$ which fixes $H$ setwise, $\phi(H)=H$.
\[
\out^H(G)=\{\Phi\in\out(H):\text{ there exists }\phi\in\Phi \text{ such that }\phi(H)=H\}
\]
Theorem~\ref{thm:Prelim1} gives, under certain conditions, the isomorphism class of this subgroup up to finite index. We write $A\leq_f B$ to mean that $A$ is a finite index subgroup of $B$.

\begin{theorem}\label{thm:Prelim1}
Let $G$ be an inner $HNN$-extension of $H$ with associated subgroup $K\lneq H$.
If $V$ is a subgroup of $H$ such that $K\leq V\leq N_H(K)$ and such that $V\cap Z(H)=1$ then $V/K$ embeds into $\out^H(G)$.
In addition, if $V\leq_f N_H(K)$ and if both $\out(H)$ and $C_H(K)$ are finite then this embedding is with finite index.
\end{theorem}

\begin{proof}
Let $\out_H(G)$ denote the subgroup of $\out(G)$ consisting of those outer automorphisms $\Phi$ with a representative $\phi$ which fixes $H$ setwise and which sends $t$ to a word containing precisely one $t$-term. The result holds for $\out_H(G)$ in place of $\out^H(G)$ \cite[Theorem A \& Lemma 5.2]{logan2015HNN}. Then $\out_H(G)=\out^H(G)$ by a result of M. Pettet \cite[Lemma 2.6]{pettet1999automorphism}.
\end{proof}

\p{Second preliminary result} The second result applied in Theorem~\ref{thm:maintheorem} is a criterion for residual finiteness of inner HNN-extensions. Ate\c{s}-Logan-Pride actually prove a more general version of this result \cite{AtesPride}.
We use the fact that a finite index subgroup $F$ of a group $G$ is residually finite if and only if $G$ is residually finite implicitly throughout the proof of this theorem. To prove this equivalence, note that subgroups of residually finite groups are clearly residually finite, while for the other direction re-write the definition of a residually finite group using normal subgroups (corresponding to the kernels of the homomorphisms $\phi_g$), and note that every finite index subgroup of $F$ contains a finite index subgroup which is normal in $G$.

\begin{proposition}[Ate\c{s}-Logan-Pride \cite{AtesPride}]
\label{thm:Prelim2}
Let $G$ be an inner $HNN$-extension of a group $H$ with non-trivial associated subgroup $K\lneq H$.
Suppose $H$ is finitely generated and residually finite, and suppose that $N_H(K)$ has finite index in $H$. Then $G$ is residually finite if and only if $N_H(K)/K$ is residually finite.
\end{proposition}

Our application of Proposition~\ref{thm:Prelim2} only uses the ``if'' direction, and not the ``only if'' direction.

\begin{proof}
Firstly, $N_H(K)/K$ embeds into $\aut(G)$ \cite[Proposition 5.3]{logan2015HNN}, hence $G$ is residually finite only if $N_H(K)/K$ is residually finite \cite{Baumslag}.

For the other direction, note that the HNN-extension $G$ is residually finite if for all finite sets $\{g_1, \ldots, g_n\}$ with $g_i\in H\setminus K$ there exists some finite index normal subgroup $N$ of $H$, $N\unlhd_f H$, such that $g_iK\cap N$ is empty for all $i\in\{1, \ldots, n\}$ \cite[Lemma 4.4]{BaumslagTretkoff}. We prove that this condition holds under the conditions of this lemma. To do this, we find for each such $g_i$ a normal subgroup $N_i$ of finite index in $H$ such that $g_iK\cap N_i$ is empty. Then, the finite-index subgroup $N:=\cap N_i$ has the required properties. There are two cases: $g_i\not\in N_H(K)$, and $g_i\in N_H(K)$.

Suppose $g_i\not\in N_H(K)$. Take the normal subgroup $N_i$ to be the intersection of the (finitely many) conjugates of $N_H(K)$. Then $hK\cap N_i$ is non-empty if and only if $h\in N_H(K)$, and hence $g_iK\cap N_i$ is empty.

Suppose $g_i\in N_H(K)$. Then $g_iK\neq K$ and because $N_H(K)/K$ is residually finite there exists a map $\psi_i: N_H(K)/K\rightarrow F_i$, such that $F_i$ is finite and $g_iK$ is not contained in the kernel of $\psi_i$. Therefore, there exists a map $\widetilde{\psi_i}: N_H(K)\rightarrow N_H(K)/K\xrightarrow{\psi_i} F_i$ such that $g_i$ is not contained in the kernel of $\widetilde{\psi_i}$, and take $N_i$ to be the kernel of the map $\widetilde{\psi_i}$. Then, $g_iK\cap N_i$ is empty by construction.
\end{proof}

\section{The proof of the main result}
\label{sec:main}

In this section we prove Theorems~\ref{corol:mainresult}~and~\ref{thm:maintheorem}. Recall that Theorem~\ref{corol:mainresult} is the main result of this paper.

\begin{thmletter}
\label{thm:maintheorem}
Fix a group $H$ such that $H$ is
\begin{enumerate}
\item\label{list:Hyp} hyperbolic,
\item\label{list:rf} residually finite, and
\item\label{list:large} large, (that is, $H$ contains a finite index subgroup $V$ which surjects onto $F_2$),
\end{enumerate}
and such that $H$ has
\begin{enumerate}[resume]
\item\label{list:FA} Serre's property FA, and
\item\label{list:tfSub} a torsion-free subgroup $U$ of finite index.
\end{enumerate}
Then every finitely-generated group $Q$ can be embedded as a finite index subgroup of the outer automorphism group of an HNN-extension $G_Q$ of $H$, where $G_Q$ is residually finite if $Q$ is residually finite.
\end{thmletter}

Note that (\ref{list:Hyp}) implies (\ref{list:rf}) if and only if (\ref{list:Hyp}) implies (\ref{list:tfSub}) \cite{kapovich2000equivalence}.

\begin{proof}
We give the construction, and then we prove that the required properties hold.

The group $G_Q$ is an inner HNN-extension, $G_Q=\langle H, t; k^t=k, k\in K\rangle$. Specifying the associated subgroup $K$ completes the construction. Let $N$ be a subgroup of $H$ such that $V/N\cong F_2$, with $V$ as in the statement of the theorem. Note that we can assume $V$ is torsion-free, as for $U$ the torsion-free subgroup of finite index the image of $V\cap U$ under the map induced by $N$ is free and non-abelian, so rewrite $V:=V\cap U$. Then, for every natural number $n$ it holds that $H$ contains a torsion-free finite-index subgroup $V_n$ which maps onto $F_n$, which can be seen by applying the correspondence theorem to the fact that the free group on two-generators contain finite-index free subgroups of arbitrary rank.

Let $Q$ be a finitely generated group. Then take a presentation $\langle X; \mathbf{r}\rangle$ of $Q$ with $2\leq|X|<\infty$ and $\mathbf{r}$ non-empty, and so $V_n$ maps onto $Q$ with $n:=|X|$. Take $K$ to be the subgroup of $V_n$ (and so of $H$) associated with the kernel of this map, so $V_n/K\cong Q$. Note that because $V_n$ has finite index in $H$,
we have that $V_n\leq_f N_H(K)\leq_f H$.

We now prove that the required properties hold. As $N_H(K)$ has finite index in $H$, Proposition~\ref{thm:Prelim2} implies that $G_Q$ is residually finite if $Q$ is residually finite. We now prove that $Q$ can be embedded as a finite index subgroup into $\out(G_Q)$. We show that the conditions of Theorem~\ref{thm:Prelim1} are satisfied, with $V:=V_n$, and so $Q$ embeds with finite index into $\out^H(G_Q)$.
The result then follows because $H$ having has Serre's property FA implies that $\out^H(G_Q)=\out(G_Q)$ \cite[Lemma~2.1]{logan2015HNN}. So, $\out(H)$ is finite as the base group $H$ is hyperbolic group with Serre's property FA \cite{levitt2005automorphisms}. Now, $K$ is non-cyclic because the map $V_n\rightarrow V_n/K$ factors through a non-cyclic free group (by assumption the set of relators $\mathbf{r}$ in the presentation for $Q$ is non-empty), and so $C_H(K)$ is finite as $H$ is hyperbolic.
By construction we have $K\leq V_n\leq_f N_H(K)$, and finally $V_n\cap Z(H)$ as $V_n$ is torsion-free by construction while $Z(H)$ is finite as $H$ is hyperbolic.
\end{proof}

We now prove Corollaries~\ref{corol:triangle}~and~\ref{corol:random}.

\begin{proof}[Proof of Corollary~\ref{corol:triangle}]
For $H$ a hyperbolic triangle group the properties (\ref{list:rf})--(\ref{list:tfSub}) are well-known to hold 
\cite{baumslag1987generalized} 
\cite{trees} 
\cite{feuer1971torsion}
.
\end{proof}

\begin{proof}[Proof of Corollary~\ref{corol:random}]
The required properties follow from Agol's theorem \cite{agol2012virtual}.
\end{proof}

We now prove the main result of this paper, Theorem~\ref{corol:mainresult}. Recall that by a \emph{continuum} we mean a set of cardinality $2^{\aleph_0}$($=|\mathbb{R}|$).

\begin{proof}[Proof of Theorem~\ref{corol:mainresult}]
Begin by noting that there exists a continuum of finitely generated, residually finite groups, and hence there is a set $\mathcal{Q}$, with cardinality the continuum, of such groups which are not recursively presentable~\cite[Theorem 4]{minasyan2009periodic}.
Applying Theorem~\ref{thm:maintheorem} to the set $\mathcal{Q}$, we obtain a set $\mathcal{G}=\{G_Q: Q\in\mathcal{Q}\}$ which consists of finitely generated, residually finite groups whose outer automorphism groups are finitely generated but not recursively presentable. Moreover, for $G_Q\in\mathcal{G}$, $\out(G_Q)$ has only countably many subgroups of finite index, and hence the set $\mathcal{G}$ contains a (subset consisting of a) continuum of groups with pairwise non-isomorphic outer automorphism groups.
\end{proof}

All the outer automorphism groups in Theorem~\ref{corol:mainresult} are residually finite. This leads us to the following question.
\begin{question}
Does there exist a finitely generated, non-recursively-presentable, non-residually-finite group $Q$ which can be realised as the outer automorphism group of a finitely generated, residually finite group $G_Q$?
\end{question}

\section{When $G_Q$ is finitely presented}
\label{sec:Qs}
We now prove a result on $\out(G_Q)$ for $G_Q$ finitely presented and residually finite.

\begin{theorem}
\label{thm:fpBumaginWise}
For every finitely presented, residually finite group $Q$ there exists a finitely presented, residually finite group $G_Q$ such that $Q$ embeds into $\out(G_Q)$.
\end{theorem}

\begin{proof}
A version of Rips' construction due to Wise \cite{wise2003ripsconstruction} gives a finitely presented, centerless, residually finite group $H_Q$ with a three-generated subgroup $N=\langle a, b, c\rangle$ such that $H_Q/N\cong Q$.%
\footnote{More recent work of Wise and his coauthors prove that the group $H_Q$ in Rips' original construction is also residually finite. The main practical difference is that $N$ can then be taken to be two-generated \cite{rips1982subgroups}.}
Then the HNN-extension $G_Q=\langle H_Q, t; a^t=a, b^t=b, c^t=c\rangle$ is residually finite, by Theorem~\ref{thm:Prelim1}, while $Q\cong H_Q/K$ embeds into $\out(G_Q)$ by Proposition~\ref{thm:Prelim2}, with $V:=H_Q=N_{H_Q}(K)$.
\end{proof}


Note that the groups $Q$ in Theorem~\ref{thm:fpBumaginWise} can be taken to be any group which embeds into a finitely presentable, residually finite group.

We know nothing about the embedding $Q\hookrightarrow \out(G_Q)$ in Theorem~\ref{thm:fpBumaginWise}.
Indeed, Theorem~\ref{thm:fpBumaginWise} is similar to a result of Wise, who proved the analogous theorem for finitely generated groups $G_Q$ by proving that $G/N$ embeds into $\out(N)$ \cite[Corollary 3.3]{wise2003ripsconstruction}. Bumagin and Wise altered Rips' construction to make Wise's embedding an isomorphism \cite{BumaginWise2005}. It may be possible to similarly alter the construction of Theorem~\ref{thm:fpBumaginWise} to answer the following question. Note that if $Q$ is finitely generated and $G_Q$ is finitely presented and residually finite then $Q$ must be recursively presentable \cite[Proposition 3.4]{logan2015outer}.

\begin{question}
Can every finitely presented group $Q$ be realised as the outer automorphism group of some finitely presented, residually finite group $G_Q$? And for $Q$ finitely generated and recursively presentable?
\end{question}

\bibliographystyle{amsalpha}
\bibliography{BibTexBibliography}

\providecommand{\bysame}{\leavevmode\hbox to3em{\hrulefill}\thinspace}
\providecommand{\MR}{\relax\ifhmode\unskip\space\fi MR }
\providecommand{\MRhref}[2]{%
  \href{http://www.ams.org/mathscinet-getitem?mr=#1}{#2}
}
\providecommand{\href}[2]{#2}
\begin{thebibliography}{AGM13}

\bibitem[AGM13]{agol2012virtual}
I.~Agol, D.~Groves, and J.~Manning, \emph{The virtual {Haken} conjecture}, Doc.
  Math. \textbf{18} (2013), 1045--1087.

\bibitem[ALP15]{AtesPride}
F.~Ate\c{s}, A.D. Logan, and S.J. Pride, \emph{Automata and {Zappa-Sz\`ep}
  products}, preprint in preparation (2015).

\bibitem[Bau63]{Baumslag}
G.~Baumslag, \emph{Automorphism groups of residually finite groups}, J. London
  Math. Soc. \textbf{48} (1963), 117--118.

\bibitem[BMS87]{baumslag1987generalized}
G.~Baumslag, J.W. Morgan, and P.B. Shalen, \emph{Generalized triangle groups},
  Math. Proc. Cambridge Philos. Soc. \textbf{102} (1987), no.~1, 25--31.

\bibitem[BT78]{BaumslagTretkoff}
B.~Baumslag and M.~Tretkoff, \emph{Residually finite {HNN} extensions}, Comm.
  Algebra \textbf{6} (1978), no.~2, 179--194.

\bibitem[BW05]{BumaginWise2005}
I.~Bumagin and D.T. Wise, \emph{Every group is an outer automorphism group of a
  finitely generated group}, J. Pure App. Alg. \textbf{200} (2005), no.~1,
  137--147.

\bibitem[BW15]{bridson2015triviality}
M.R. Bridson and H.~Wilton, \emph{The triviality problem for profinite
  completions}, Invent. Math. (to appear) (2015).

\bibitem[DGP11]{dahmani2011random}
F.~Dahmani, V.~Guirardel, and P.~Przytycki, \emph{Random groups do not split},
  Math. Ann. \textbf{349} (2011), no.~3, 657--673.

\bibitem[Feu71]{feuer1971torsion}
R.D. Feuer, \emph{Torsion-free subgroups of triangle groups}, Proc. Amer. Math.
  Soc. \textbf{30} (1971), no.~2, 235--240.

\bibitem[Gro96]{gromov1996geometric}
M.~Gromov, \emph{Geometric group theory, vol. 2: Asymptotic invariants of
  infinite groups}, Bull. Amer. Math. Soc \textbf{33} (1996), 0273--0979.

\bibitem[KW00]{kapovich2000equivalence}
I.~Kapovich and D.T. Wise, \emph{The equivalence of some residual properties of
  word-hyperbolic groups}, J. Algebra \textbf{223} (2000), no.~2, 562--583.

\bibitem[Lev05]{levitt2005automorphisms}
G.~Levitt, \emph{Automorphisms of hyperbolic groups and graphs of groups},
  Geom. Ded. \textbf{114} (2005), no.~1, 49--70.

\bibitem[Log15a]{logan2015HNN}
A.D. Logan, \emph{The {Bass-Jiang} group for automorphism-induced
  {HNN}-extensions}, arXiv:1509.01847 (2015).

\bibitem[Log15b]{logan2015outer}
\bysame, \emph{On the outer automorphism groups of finitely generated,
  residually finite groups}, J. Algebra \textbf{423} (2015), 890--901.

\bibitem[Mal40]{Malcev1940}
A.~Mal'cev, \emph{On isomorphic matrix representations of infinite groups},
  Rec. Math. [Mat. Sbornik] N.S. \textbf{8(50)} (1940), no.~3, 405--422.

\bibitem[MOS09]{minasyan2009periodic}
A.~Minasyan, A.Y. Olshanskii, and D.~Sonkin, \emph{Periodic quotients of
  hyperbolic and large groups}, Groups Geom. Dyn. \textbf{3} (2009), no.~3,
  423--452.

\bibitem[Oll05]{ollivier2005january}
Y.~Ollivier, \emph{A january 2005 invitation to random groups}, vol.~10,
  Sociedade brasileira de matem{\'a}tica, 2005.

\bibitem[OW11]{ollivier2011cubulating}
Y.~Ollivier and D.~Wise, \emph{Cubulating random groups at density less than
  1/6}, Trans. Amer. Math. Soc. \textbf{363} (2011), no.~9, 4701--4733.

\bibitem[Pet99]{pettet1999automorphism}
M.R. Pettet, \emph{The automorphism group of a graph product of groups}, Comm.
  Algebra \textbf{27} (1999), no.~10, 4691--4708.

\bibitem[Rip82]{rips1982subgroups}
E.~Rips, \emph{Subgroups of small cancellation groups}, Bull. London Math. Soc
  \textbf{14} (1982), no.~1, 45--47.

\bibitem[Slo81]{Slobodskoi1981undecidability}
A.~M. Slobodsko\v\i, \emph{Undecidability of the universal theory of finite
  groups}, Algebra i Logika \textbf{20} (1981), no.~2, 207--230, 251.

\bibitem[SS03]{trees}
J.P. Serre and J.~Stilwell, \emph{Trees}, Springer Monographs in Mathematics,
  Springer, 2003.

\bibitem[Wis03]{wise2003ripsconstruction}
D.T. Wise, \emph{A residually finite version of {Rips's} construction}, Bull.
  London Math. Soc. \textbf{35} (2003), no.~01, 23--29.

\end{thebibliography}

\end{document}